\documentclass[12pt,a4paper]{amsart}
\usepackage{fullpage,color}
\usepackage[utf8]{inputenc}
\usepackage{amsmath,amsfonts,amssymb,amsthm,amscd}
\usepackage{graphicx}
\usepackage{verbatim}

\theoremstyle{plain}
\newtheorem{theorem}{Theorem}
\newtheorem{proposition}{Proposition}
\newtheorem{lemma}{Lemma}
\newtheorem{corollary}[theorem]{Corollary}

\theoremstyle{definition}

\newtheorem*{remark*}{Remark}

\newcommand{\Z}{\mathbb Z}

\usepackage{hyperref}

\date{\today}

\title{Euler's divergent series in arithmetic progressions}

\author{Anne-Maria Ernvall-Hyt\"onen}
\address{Anne-Maria Ernvall-Hyt\"onen, Matematik och Statistik, {\AA}bo Akademi University, Domkyrkotorget 1, 20500 {\AA}bo, Finland}
\email{anne-maria.ernvall-hytonen@abo.fi}

\author{Tapani Matala-aho}
\address{Tapani Matala-aho, Matematiikka, PL 8000, 90014 Oulun yliopisto, Finland}
\email{tapani.matala-aho@oulu.fi}

\author{Louna Sepp\"al\"a}
\address{Louna Sepp\"al\"a, Matematiikka, PL 8000, 90014 Oulun yliopisto, Finland}
\email{louna.seppala@oulu.fi}

\thanks{The work of Louna Sepp\"al\"a was supported by the University of Oulu Scholarship Foundation and the Vilho, Yrj\"o and Kalle V\"ais\"al\"a Foundation.}



\begin{document}

\begin{abstract}
Let $\xi$ and $m$ be integers satisfying $\xi\ne 0$ and $m\ge 3$. 
We show that for any given integers $a$ and $b$, $b \neq 0$, there are $\frac{\varphi(m)}{2}$ reduced residue classes modulo $m$ each containing 
infinitely many primes $p$ such that $a-bF_p(\xi) \ne 0$, where $F_p(\xi)=\sum_{n=0}^\infty n!\xi^n$ is the $p$-adic evaluation of Euler's factorial series at the point $\xi$.
\end{abstract}

\maketitle

\section{Introduction and results}

Euler's factorial series is defined as the sum
\begin{equation}\label{hyperseries}
F(z):={}_2F_0(1,1\mid z)=\sum_{n=0}^\infty n!z^n.
\end{equation}
It is clear that in the standard Archimedean metric, it only converges when $z=0$. 
In the case of the $p$-adic metric, however, the situation changes drastically.
For a prime $p$ the normalization $|p|_p=p^{-1}$ gives the usual $p$-adic absolute value.
The $p$-adic completion of the rationals $\mathbb{Q}$ with respect to the metric $|\cdot|_p$ is denoted by $\mathbb{Q}_p$.
Now the series \eqref{hyperseries} converges in the unit disc $\left\{z\in\mathbb{Q}_p \; | \; |z|_p\le 1 \right\}$ and consequently defines a function $F_p$
in that disc by the values $F_p(z):=\sum_{n=0}^\infty n!z^n$.

We are interested in arithmetical properties of the values $F_p(\xi)$ of Euler's factorial series at non-zero integer points
$\xi \in \Z \setminus \{0\}$. It is an open question whether e.g.~the values $F_p(\pm 1)$ are irrational or even non-zero, 
which is why it has become customary to study global relations. Let $P(x)\in \mathbb{Z}[x]$, $d := \deg P(x)\ge 1$.
For a given $\xi$, a \emph{global relation of degree $d$} is any polynomial identity $P(F_p(\xi))=0$ which is satisfied for all the primes $p$ such that 
$F_p(\xi)$ is defined.
There are several works considering global relations of the series \eqref{hyperseries} and its generalizations,
including generalised hypergeometric series; see e.g.~\cite{BCY04}, \cite{Chirskii1992}, \cite{Chirskii2014}.

In the recent works \cite{Chirskii2015} and \cite{TapaniWadim} the authors investigated first degree global relations
for Euler's factorial series. Chirski\u \i \ proved (with our notation)

\begin{proposition}\label{Chirskiiglobal}\cite{Chirskii2015}
There exist infinitely many primes $p$ such that $F_p(1) \ne 0$.
\end{proposition}

The authors in \cite{TapaniWadim} proved

\begin{proposition}\label{TAWA}\cite{TapaniWadim}
Given $\xi\in\mathbb{Z}\setminus\{0\}$, let $R\subseteq \mathbb{P}$ be  such that
\begin{equation}\label{riistajaehto}
\limsup_{n\to\infty} c^nn!\prod_{p\in R}|n!|_p^2=0,
\quad\text{where}\; c=c(\xi;R):=4|\xi|\prod_{p\in R}|\xi|_p^2.
\end{equation}
Then either there exists a prime $p\in R$ for which $F_p(\xi)$ is irrational, or there are two distinct primes $p,q\in R$
such that $F_p(\xi)\ne F_q(\xi)$ \textup(while $F_p(\xi), F_q(\xi)\in \mathbb{Q}$\textup).
\end{proposition}

From now on let $\Lambda(x)=a-bx\in\mathbb{Z}[x]$, $b\ne 0$. Note that Proposition \ref{Chirskiiglobal}
corresponds to the case $\Lambda(x)=x$ and Proposition \ref{TAWA} to the case $\Lambda(x)=a-bx$.
Based on Proposition \ref{TAWA}, we are able to prove a more extensive result, Theorem \ref{AB}, which we shall formulate by using the polynomial $\Lambda(x)=a-bx$.

For the rest of the work we assume $\xi\in\mathbb{Z}\setminus\{0\}$.

\begin{theorem}\label{AB}
Let $T\subseteq\mathbb{P}$ be a subset of primes such that the set $T\setminus S$ satisfies condition \eqref{riistajaehto} for any
finite subset $S$ of $T$.
Then there exist infinitely many primes $p\in T$ such that $\Lambda(F_p(\xi))\ne 0$.
\end{theorem}

\begin{proof}
Define
\begin{equation*}
\mathcal{R}:=\left\{R\subseteq \mathbb{P} \; | \; R\ \text{satisfies condition}\ \eqref{riistajaehto} \right\}.
\end{equation*}
If $R\in \mathcal{R}$, the proof of Proposition \ref{TAWA} in \cite{TapaniWadim}
shows that there exists at least one prime $p\in R$ such that $\Lambda(F_p(\xi)) \ne 0$.

Take now a subset $T$ of primes such that $T\setminus S \in \mathcal{R}$ for any finite subset $S$ of $T$. 
Define a new set 
\begin{equation*}
A:= \left\{p\in T\; |\; \Lambda(F_p(\xi)) \ne 0 \right\}.
\end{equation*}
If the set $A$ is finite, then $T\setminus A\in \mathcal{R}$ by assumption.
But
\begin{equation*}
T\setminus A=\{p\in T\; |\; \Lambda(F_p(\xi))= 0 \}.
\end{equation*}
This is a contradiction, and thus $\# A=\infty$.
\end{proof}

From Theorem \ref{AB} it follows that for any set $R\in \mathcal{R}$ satisfying the assumptions of Theorem \ref{AB}, 
there exist infinitely many primes $p\in R$ such that $F_p(\xi) \ne \frac{a}{b}$.
In particular, if we take $T=\mathbb{P}$, then we see immediately that any 
$\mathbb{P}\setminus S \in \mathcal{R}$, if $S$ is a finite set of primes.
Thus, Theorem \ref{AB} implies the following corollary.

\begin{corollary}\label{Corkaikkip}
Let $\frac{a}{b} \in \mathbb{Q}$ be given. 
Then there exist infinitely many primes $p\in \mathbb{P}$ such that $F_p(\xi) \ne \frac{a}{b}$. 
\end{corollary}

This still seems to be far from implying irrationality, for the prime $p=p_{\Lambda}\in R$ for which $F_p(\xi) \ne \frac{a}{b}$, 
may depend on the polynomial $\Lambda(x)=a-bx$. We note that there are results, see e.g. \cite{BCY04}, \cite{Chirskii1992}, from which Corollary \ref{Corkaikkip} follows, but the proof presented here is different from these earlier works. 
As will be seen shortly, we may also considerably diminish the prime number set where Corollary \ref{Corkaikkip} is still valid.

The question rises whether, for example, the reduced residue system modulo $m$, $m\in\mathbb{Z}_{\ge 3}$, 
could produce examples of prime subsets satisfying condition \eqref{riistajaehto}.
Indeed, that is the case, as will be demonstrated in the following theorem, the main result of this paper.

Let $m \in \Z_{\ge 3}$ and denote $\overline{a} := \{ a+km \; | \; k \in \Z \}$. We write $\overline{a}_1, \ldots, \overline{a}_{\varphi (m)}$ 
for the $\varphi (m)$ residue classes in the reduced residue system modulo $m$. Dirichlet's theorem about primes in arithmetic progressions
tells that each of these classes contains infinitely many prime numbers.

\begin{theorem}\label{epatasa}
Let $m\in\mathbb{Z}_{\ge 3}$ be a given integer. Assume that $R=\bigcup_{j=1}^r \left( \overline{a}_{i_j} \cap \mathbb{P} \right)$ 
is any union of the primes in $r$ residue classes in the reduced residue system modulo $m$, where $r > \frac{\varphi(m)}{2}$. 
Then there are infinitely many primes $p \in R$ such that $\Lambda(F_p(\xi))\ne 0$.
\end{theorem}

Observe that the "global relation set" $G_{\Lambda}:=\{p\in\mathbb{P}\;|\; \Lambda(F_p(\xi))=0\}$ cannot be too big:
By Theorem \ref{AB}, the set $G_{\Lambda}$ obviously cannot satisfy condition \eqref{riistajaehto}.
Theorem \ref{epatasa} shows that the set $G_\Lambda$ cannot contain the primes of more than half of the reduced residue classes modulo $m$.
Therefore $F_p(\xi) \ne \frac{a}{b}$ holds---in the above sense---for at least "half" of all the primes $p\in\mathbb{P}$.

Theorem \ref{epatasa} implies that there is a residue class modulo $m$ containing infinitely many primes $p$ for which $\Lambda(F_p(\xi))\ne 0$.
Since the union $R$ may be chosen arbitrarily, we actually obtain: 

\begin{corollary}
Let $m\in\mathbb{Z}_{\ge 3}$ be a given integer. There are $\frac{\varphi(m)}{2}$ reduced residue classes modulo $m$ each containing 
infinitely many primes $p$ such that $F_p(\xi) \ne \frac{a}{b}$. 
\end{corollary}

\begin{proof}
By Theorem \ref{epatasa}, the union $\bigcup_{i=1}^{\frac{\varphi(m)}{2}+1} \overline{a}_i$ contains infinitely many primes $p$ such that 
$\Lambda(F_p(\xi))\ne 0$. Thus one of the residue reduced residue classes $\overline{a}_1, \ldots, \overline{a}_{\frac{\varphi(m)}{2}+1}$ 
must contain infinitely many such primes---suppose it is $\overline{a}_1$. Now we may apply Theorem \ref{epatasa} again to the union 
$\bigcup_{i=2}^{\frac{\varphi(m)}{2}+2} \overline{a}_i$, and without loss of generality assume that this time the residue class containing 
infinitely many of those certain primes is $\overline{a}_2$. This procedure can be repeated $\frac{\varphi (m)}{2}$ times, 
from which the assertion follows.
\end{proof}

The number $\frac{\varphi(m)}{2}$ is always an integer because $\varphi(m)$ is even when $m\ge 3$. The case $m=2$ is of no interest 
because all the primes except the prime $2$ are in the same residue class.
 
Finally, assuming the generalised Riemann hypothesis (GRH), we can say something slightly different, namely that in any collection of $\frac{\varphi(m)}{2}$ residue classes in the reduced residue system modulo $m$ there is at least one prime satisfying the condition $\Lambda(F_p(\xi))\ne 0$, but now for a $\xi$ satisfying the condition of Theorem 5. The previous theorem gave us the existence of infinitely many such primes in some $\frac{\varphi(m)}{2}$ residue classes. Now we can prove that also in the complement of the previous $\frac{\varphi(m)}{2}$ residue classes, there must be at least one prime satisfying the condition.

\begin{theorem}\label{tasa}
Assume the GRH. Let $m\in\mathbb{Z}_{\ge 3}$ be a given integer. Assume that $R=\bigcup_{j=1}^{\varphi(m)/2} \left( \overline{a}_{i_j} \cap \mathbb{P} \right)$ 
is any union of the primes in $\frac{\varphi(m)}{2}$ residue classes in the reduced residue system modulo $m$. Then there is a value  $d_m$ such that if $\xi$ is any non-zero integer satisfying the bound
\[
4|\xi|\prod_{p\in R}|\xi|_p^2<d_m,
\]
there exists a prime $p\in R$ for which $\Lambda(F_p(\xi))\ne 0$.
\end{theorem}

\section{Preliminaries}

From \cite{theta} and \cite[Ch. 20]{Davenport} we recall the estimates for the functions 
$$\theta(x;m,a)=\sum_{\substack{p\equiv a \bmod m\\ p\in \mathbb{P},\  p\leq x}}\log p$$ 
and 
$$\psi(x;m,a)=\sum_{\substack{p^k\equiv a \bmod m\\ p\in \mathbb{P},\ p^k \leq x}}\log p$$ 

Very recently, Bennett, Martin, O'Bryant and Rechnitzer \cite{theta} proved that if $q \in \mathbb{Z}_{\ge 3}$ and $a$ is an integer coprime to $q$, then there exist positive constants $c_{\theta}(q)$ and $x_{\theta} (q)$ such that 
\[
\left|\theta(x;q,a)-\frac{x}{\varphi(q)}\right|<c_{\theta}(q)\frac{x}{\log x}
\]
for all $x \ge x_\theta (q)$, and they even gave some values and bounds for these constants. Immediately, it follows that a similar bound holds for all $x\geq 2$ for some (not necessarily good) value of $c_{\theta}(q)$.

The error term can be substantially improved if we assume the generalised Riemann hypothesis. See e.g. Davenport's book \cite[Ch. 20]{Davenport} for a good introduction into the topic. There one can also find the following useful bound: 
Assuming the GRH, we have 
\[
\psi(x;m,a)=\frac{x}{\varphi(m)}+O\left(\sqrt{x}\log^2x\right),
\]
where the $O$ term depends on $m$.

Further recall that the second bound immediately gives the same bound for the function $\theta(x;q,a)$.

The following lemma may be known but we state and prove it for the sake of completeness. The main term is clear by the distribution of primes in residue classes. The critical part is that the contribution coming from the other terms is not too large.

\begin{lemma}\label{lemma2}
Let $m\in\mathbb{Z}_{\ge 3}$ be a given integer and assume $\gcd (a,m) =1$. Then
\[
\log \left(\prod_{p\equiv a\pmod m} |n!|_p\right)=-\frac{n\log n}{\varphi(m)}+O(n\log \log n).
\]
Further, assuming the GRH, we obtain
\[
\log \left(\prod_{p\equiv a\pmod m} |n!|_p\right)=-\frac{n\log n}{\varphi(m)}+O(n).
\]
\end{lemma}

\begin{proof} It is well known that $n!$ is divisible by a prime $p$ exactly $\sum_{k=1}^{\infty}\left\lfloor \frac{n}{p^k}\right\rfloor$ times. 
This term can be estimated in the following way (see \cite{KTT}):
\[
\frac{n}{p-1}-\frac{\log n}{\log p}-1 \leq \sum_{k=1}^{\infty}\left\lfloor \frac{n}{p^k}\right\rfloor \leq \frac{n-1}{p-1}.
\]
Therefore, the $p$-adic valuation of $n!$ satisfies the bound
\[
p^{-\frac{n}{p-1}}\leq |n!|_p\leq p^{-\frac{n}{p-1}+\frac{\log n}{\log p}+1}.
\]
Let us start by treating the contribution coming from the term $p^{-\frac{n}{p-1}}$. First use the Abel summation formula (see \cite[Theorem 4.2]{Apostol}):
\begin{equation}\label{Abel}
\begin{split}
-\sum_{\substack{p\leq n \\ p\equiv a(m)}}\log p^{-\frac{n}{p-1}} &=n\sum_{\substack{p\leq n \\ p\equiv a(m)}}\frac{\log p}{p-1}\\
&=\frac{n}{n-1} \sum_{\substack{p\leq n \\ p\equiv a(m)}} \log p + n\int_2^n \left(\sum_{\substack{p\leq x \\ p\equiv a(m)}} \log p \right)\cdot \frac{1}{(x-1)^2} \mathrm{d}x.
\end{split}
\end{equation}

Without assuming the GRH, this can be written as
\begin{multline*}
-\sum_{\substack{p\leq n \\ p\equiv a(m)}}\log p^{-\frac{n}{p-1}}=\frac{n}{n-1}\sum_{\substack{p\leq n \\ p\equiv a(m)}}\log p+n\int_2^n\left(\sum_{\substack{p\leq x \\ p\equiv a(m)}}\log p\right)\cdot \frac{1}{(x-1)^2}\mathrm{d}x\\
=\frac{n}{n-1}\left(\frac{n}{\varphi(m)}+O\left(\frac{n}{\log n}\right)\right)+n\int_2^{n}\left(\frac{x}{\varphi(m)}+O\left(\frac{x}{\log x}\right)\right)\cdot \frac{1}{(x-1)^2}\mathrm{d}x\\
=\frac{n}{\varphi(m)}+O\left(\frac{n}{\log n}\right)+\frac{n}{\varphi(m)}\int_2^n \frac{x\mathrm{d}x}{(x-1)^2}+O\left(n\int_2^n\frac{x\mathrm{d}x}{(x-1)^2\log x}\right)\\
=\frac{n\log n}{\varphi(m)}+O(n)+ O \left( n\int_2^n\frac{(x-1)\mathrm{d}x}{(x-1)^2\log x}+n \int_2^n\frac{\mathrm{d}x}{(x-1)^2\log x} \right)\\
=\frac{n\log n}{\varphi(m)}+O(n)+O \left( n\int_2^3\frac{\mathrm{d}x}{(x-1)\log x}+n \int_3^n\frac{\mathrm{d}x}{(x-1)\log x} \right).
\end{multline*}
The last term can be estimated
\[
n\int_3^n\frac{\mathrm{d}x}{(x-1)\log x}\leq n\int_2^{n-1}\frac{\mathrm{d}x}{x\log x}=O(n\log \log n).
\]
Hence
\[
-\sum_{\substack{p\leq n \\ p\equiv a(m)}}\log p^{-\frac{n}{p-1}}=\frac{n\log n}{\varphi(m)}+O(n\log \log n).
\]

Assuming the GRH, \eqref{Abel} can be written as
\begin{multline*}
-\sum_{\substack{p\leq n \\ p\equiv a(m)}}\log p^{-\frac{n}{p-1}}=\frac{n}{n-1}\sum_{\substack{p\leq n \\ p\equiv a(m)}}\log p+n\int_2^n\left(\sum_{\substack{p\leq x \\ p\equiv a(m)}}\log p\right)\cdot \frac{1}{(x-1)^2}\mathrm{d}x\\ 
=\frac{n}{n-1}\left(\frac{n}{\varphi(m)}+O\left(\sqrt{n}\log^2 n\right)\right)+n\int_2^{n}\left(\frac{x}{\varphi(m)}+O\left(\sqrt{x}\log^2x\right)\right)\cdot \frac{1}{(x-1)^2}\mathrm{d}x\\  
=\frac{n}{\varphi(m)}+O\left(\sqrt{n}\log^2 n\right)+\frac{n}{\varphi(m)}\int_1^{n-1}\frac{\mathrm{d}x}{x}+\frac{n}{\varphi(m)}\int_1^{n-1}\frac{\mathrm{d}x}{x^2}+O(n)=\frac{n\log n}{\varphi(m)}+O(n).\end{multline*}

Let us now look at the term $p^{\log n/\log p+1}$:
\[
\sum_{p\leq n,\ p\equiv a(m)}\log p^{\frac{\log n}{\log p}+1}=\sum_{p\leq n, \ p\equiv a(m)} (\log n+\log p)=O(n).
\]

Combining these, the contribution coming from all $p\equiv a\pmod m$ is
\[
\log\left(\prod_{\substack{p\leq n \\ p\equiv a(m)}} |n!|_p\right)=\sum_{\substack{p\leq n \\ p\equiv a(m)}}\log p^{-\frac{n}{p-1}+O\left(\frac{\log n}{\log p}+1\right)}=\frac{n\log n}{\varphi(m)}+O(n\log \log n),
\]
and assuming the generalised Riemann hypothesis, it will be
\[
\log\left(\prod_{\substack{p\leq n \\ p\equiv a(m)}} |n!|_p\right)=\sum_{\substack{p\leq n \\ p\equiv a(m)}}\log p^{-\frac{n}{p-1}+O\left(\frac{\log n}{\log p}+1\right)}=\frac{n\log n}{\varphi(m)}+O(n).
\]
\end{proof}

\section{Proofs of Theorems \ref{epatasa} and \ref{tasa}}

Let us first prove Theorem \ref{epatasa}.

\begin{proof}[Proof of Theorem \ref{epatasa}]
We shall use Theorem \ref{AB}.
By assumption, $R=\bigcup_{j=1}^r \left( \overline{a}_{i_j} \cap \mathbb{P} \right)$ is a union of the primes in $r$ residue classes 
$\overline{a}_{i_1}, \ldots, \overline{a}_{i_r}$ in the reduced residue system modulo $m$, where $r > \frac{\varphi(m)}{2}$. 
It suffices to prove that for any non-zero integer $\xi$ and a finite subset $S=\{p_1,\ldots,p_k\} \subseteq R$, the condition
\[
\lim\sup_{n\rightarrow \infty} c_0^{n}n! \prod_{p\in R\setminus S}|n!|_p^2=0, \quad c_0=4|\xi|\prod_{p\in R\setminus S}|\xi|_p^2,
\]
is satisfied. Thus we are led to study the expression
\begin{equation}\label{logeq}
\log \left(c_0^{n}n! \prod_{p\in R \setminus S}|n!|_p^2\right)=n\log c_0+\log n!+2\sum_{p\in R \setminus S} \log |n!|_p.
\end{equation}
Here
\begin{equation}\label{S-summa}
2\sum_{p\in S}\log |n!|_p=2\sum_{i=1}^k\log |n!|_{p_i}=O(n)
\end{equation}
when $n$ grows for any given fixed $S$. The constant implied by the $O$-term may be arbitrarily large but it is constant in the $n$ aspect.

Recall the Stirling formula (see e.g. \cite{stirling}, formula 6.1.38):
\[
\log n!=\log \sqrt{2\pi}+\left(n+\frac{1}{2}\right)\log n-n+\frac{\theta(n)}{12},
\]
where $0<\theta(n)<1$. The above can be further simplified to $\log n!=n\log n+O(n)$. Combining this with \eqref{logeq}, \eqref{S-summa}, and Lemma \ref{lemma2}, we get
\begin{align*}
\log \left(c_0^{n}n! \prod_{p\in R\setminus S}|n!|_p^2\right)&=n\log n+O(n)+2\sum_{j=1}^r \sum_{p \in \overline{a}_{i_j}} \log |n!|_p - 2 \sum_{i=1}^k \log |n!|_{p_i}\\
&=n\log n+O(n)-\frac{2rn\log n}{\varphi(m)}+O(n\log \log n)\\
&=n\log n\left(1-\frac{2r}{\varphi(m)}\right)+O(n)+O(n\log \log n) \rightarrow -\infty
\end{align*}
as $n \to \infty$ because the coefficient $1-\frac{2r}{\varphi(n)}$ of the main term is negative.
The result follows from Theorem \ref{AB}.
\end{proof}

Let us now move to the proof of Theorem \ref{tasa}.
\begin{proof}[Proof of Theorem \ref{tasa}]
Here we use Proposition \ref{TAWA}, so we need to check condition \eqref{riistajaehto} with $c=4|\xi|\prod_{p\in R}|\xi|_p^2$. When looking at the terms in 
\[
\log \left(c^{n}n! \prod_{p\in R}|n!|_p^2\right)=n\log c+\log n!+2\sum_{p\in R} \log |n!|_p,
\]
we again use Stirling's formula and the bound for $|n!|_p$ as earlier. Now the main terms cancel, so what is left is $O(n)$ for some constant depending only on $m$. Therefore, the contribution of this term can be cancelled if $c$ is sufficiently small, namely, below some $d_m$ depending only on $m$.
\end{proof}
\begin{remark*}
This proof of Theorem \ref{tasa} cannot be generalized to the case with infinitely many primes because the constant implied by the $O$-term in contribution of the arbitrary subset $S$ can be arbitrarily large, and therefore, we cannot use the argument of a term of magnitude $n$ cancelling the other terms.
\end{remark*}

\end{document}